\documentclass[10pt]{article}
\title{{ An analytic approach to special numbers and polynomials}}
\author{Grzegorz Rz\c{a}dkowski
}
\usepackage{amssymb}
\usepackage{latexsym}
\usepackage{amsthm}
\date{Faculty of Mathematics and Natural Sciences,\\ Cardinal Stefan Wyszy\'nski 
University in Warsaw,\\ Dewajtis 5, 01 - 815 Warsaw, Poland\\ g.rzadkowski@uksw.edu.pl  \\
grzerzad@gmail.com} 
\begin{document}
\maketitle
\newtheorem{lemma}{Lemma}
\newtheorem{theorem}{Theorem}
\newtheorem{remark}{Remark}
\newtheorem{statement}{Statement}
\def \bangle{ \atopwithdelims \langle \rangle}
\begin{abstract}
The purpose of this article is to present, in a simple way, an analytic approach to special numbers and polynomials.
The approach is based on the derivative polynomials. The paper is, to some extent, a review article, although it contains some new elements. In particular, it seems that some integral representations for Bernoulli numbers and Bernoulli polynomials can be seen as new.
\end{abstract}
\section{Introduction}
Let $u=u(z)$ be a holomorphic function defined in a domain $D_{u} \subset \mathbb{C}$ which fulfills the Riccati differential equation with constant coefficients
\begin{equation}\label{R}
	u'=r(u-a)(u-b)
\end{equation}
where $r,a,b$ are real or complex numbers $r\neq 0, \; a \neq b$. Let $v=v(z)$ be a holomorphic function defined in a domain $D_{v} \subset \mathbb{C}$ which is related with $u(z)$ and fulfills the following differential equation
\begin{equation}\label{R2}
	v'=rv\left(u-\frac{a+b}{2}\right),
\end{equation}
where $a,b,r,u(z)$ are as in (\ref{R}).\\
Examples of such pairs of functions and equations are:
\begin{enumerate}
	\item $u(z)=\tan z, \quad u'(z)=u^{2}+1,\quad v(z)= \sec z, \quad v'=vu,$
	\item $u(z)=\tanh z, \quad u'(z)=-u^{2}+1,\quad v(z)= 1/\cosh z, \quad v'=-vu,$
	\item $u(z)=\cot z, \quad u'(z)=-u^{2}-1,\quad v(z)= \csc z, \quad v'=-vu,$
	\item $u(z)=\coth z, \quad x'(z)=u^{2}-1,\quad v(z)= 1/\sinh z, \quad v'=vu,$
	\item $u(z)=1/(1+e^{ z}), \quad u'(z)=u^{2}-u, \quad v(z)= e^{z/2}/(1+e^{ z}), \quad \\ v'=v(u-1/2),$\label{l}
	\item $u(z)=1/(1+e^{ -z}), \quad u'(z)=-u^{2}+u, \quad v(z)= e^{-z/2}/(1+e^{-z}), \\ v'=-v(u-1/2),$
	\item more generally the logistic function: $u(z)=q/(1+pe^{ -sz}), \quad \\ u'(z)=\frac{s}{q}(q-u)u, \quad v(z)=qe^{-sz/2}/(1+pe^{ -sz}), \quad v'(z)=\frac{s}{q}v(q/2-u)$ (with $p>1,\; q>0,\; s>0$). \label{log}
\end{enumerate}
We will consider also the following generalization of equation (\ref{R2}) 
\begin{equation}\label{R2bis}
	v'=rv\left(u-\frac{a+b}{2}+d\right),
\end{equation}
where d is a real or complex number.\\
Such system of differential equations has been investigated by Hoffman \cite{Ho} (instead of equations (\ref{R2}-\ref{R2bis}) he regarded $v'=vu$ but in each particular case, which he regarded, $a+b=0$) and by Franssens \cite{F2} (who investigated the equation $v'=-vu$)\\ 
Let $\{a_1,a_2,\ldots ,a_n\}$ be a permutation of the set $\{1,2,\ldots ,n\}$. Then $\{a_{j},a_{j+1}\}$ is an ascent of the permutation if $a_j< a_{j+1}$. 
The Eulerian number $\displaystyle  {n \bangle k} $ is defined as the number of permutations of the set $\{1,2,\ldots ,n\}$ having $k$ permutation ascents (see \cite{GKP}, p.267). For example for $n=3$ the permutation $\{1,2,3\}$ has two ascents, namely $\{1,2\}$ and $\{2,3\}$, and  $\{3,2,1\}$ has no ascents. Each of the other four permutations of the set has exactly one ascent. Thus $\displaystyle {3 \bangle 0} =1 $,  $\displaystyle {3 \bangle 1} =4 $, and $\displaystyle  {3 \bangle 2} =1 $. It is well known that Eulerian numbers satisfy the following relations:
{\setlength\arraycolsep{2pt}
\begin{eqnarray}
{n \bangle k} &=&  {n \bangle n-k-1} , \nonumber \\
 {n+1 \bangle k} &=& (k+1) {n \bangle k}
+(n-k+1) {n \bangle k-1} , \label{eq1}\\
 {n \bangle k} &=& \sum\limits_{j=0}^{k}(-1)^{j}{n+1 \choose j}(k-j+1)^{n}. \label{eq2}
\end{eqnarray}}
The Eulerian polynomial $E_{n}(x),\; n=0,1,2,\ldots $ is defined (see Comtet \cite{C})  by the formula
\begin{equation}\label{Ep}
	E_{n}(x)= \sum\limits_{k=0}^{n-1}{n \bangle k}x^{k+1}\;\;\textrm{for}\; n\ge 1,\quad E_0(x)=1.
\end{equation}
There is a slightly different definition of the Eulerian polynomial $A_{n}(x)$ (see for example Foata \cite{Fo}) i.e.,
\begin{equation}\label{Ep2}
	A_{n}(x)= \sum\limits_{k=0}^{n-1}{n \bangle k}x^{k}, \quad A_0(x)=1.
\end{equation}
Thus $E_{n}(x)=xA_{n}(x)$ for $n\ge 1$, $A_{1}(x)= A_{0}(x)=E_{0}(x)\equiv 1$.\\
The MacMahon numbers $\{M_{n,k}\}$ are defined by the recurrence formula (see \cite{M}, \cite{F})
\begin{equation}\label{MM}
	M_{n,k} = (2k-1)M_{n-1,k} + (2n-2k+1)M_{n-1,k-1},
\end{equation}
where $1\le k\le n,\;\; M(n,1)=1,\; n=1,2,\ldots $.\\
The MacMahon polynomial $M_{n}(x),\; n=0,1,2,\ldots $ is defined as follows
\begin{equation}\label{MMp}
	M_{n}(x)=\sum\limits_{k=1}^{n+1}M_{n+1,k}\:x^{k-1}.
\end{equation}
\section{Derivative polynomials}
The following theorem has been discussed during the Conference ICNAAM 2006 (September 2006) in Greece and it appears in the paper \cite{Rz}. For convenience of the reader we give it with an inductive proof. Independently the theorem has been considered and proved, with a proof based on generating functions, by Franssens \cite{F2} (see also \cite{Rz1}).
\begin{theorem}
If a function $u(z)$ satisfies equation (\ref{R}),
then the $n$th derivative of $u(z)$ can be expressed by the following formula:
\begin{equation}\label{f}
u^{(n)}(z) = r^{n}\sum\limits_{k=0}^{n-1} {n \bangle k}
(u-a)^{k+1}(u-b)^{n-k}
\end{equation}
where $n=2,3,\ldots $.
\end{theorem} 
\begin{proof} 
By (\ref{R}) we get
	\[	u''(t) = r[(u-a)+(u-b)]u'(z)=r^{2}[(u-a)(u-b)^{2}+(u-a)^{2}(u-b)],
\]
which establishes (\ref{f}) for $n=2$. Let us assume that for an integer $n\ge 2$ formula (\ref{f}) holds. Using recurrence formula (\ref{eq1}) in the last step of the following calculation we get 
{\setlength\arraycolsep{2pt}
\begin{eqnarray}
&& u^{(n+1)}(z)= r^{n}\frac{d}{dz}\sum\limits_{k=0}^{n-1} {n \bangle k}
(u-a)^{k+1}(u-b)^{n-k}\nonumber \\
&&= r^{n+1}\sum\limits_{k=0}^{n-1} {n \bangle k}
\left[(k\!+\!1)(u\!-\!a)^{k+1}(u\!-\!b)^{n-k+1}\!+\!(n\!-\!k)(u\!-\!a)^{k+2}(u\!-\!b)^{n-k}\right]\nonumber \\
&& =r^{n+1}\left[ {n \bangle 0} (u\!-\!a)(u\!-\!b)^{n+1}
+\sum\limits_{k=1}^{n-1}\left((k\!+\!1) {n \bangle k}
+(n\!-\!k\!+\!1) {n \bangle k\!-\!1} \right)\right.\nonumber \\
&&\hspace{11mm}\left. \times (u\!-\!a)^{k+1}(u\!-\!b)^{n-k+1}
+ {n \bangle n\!-\!1} (u\!-\!a)^{n+1}(u\!-\!b)\right]
\nonumber \\
&& =r^{n+1}\sum\limits_{k=0}^{n} {n\!+\!1 \bangle k}
(u-a)^{k+1}(u-b)^{n-k+1},\nonumber
\end{eqnarray}}
which ends the proof.
\end{proof}
The following two theorems (2 and 3) are connected with solutions of equations (\ref{R2}) and (\ref{R2bis}) respectively and  are due to Franssens \cite{F2}. Theorem 2 is a particular case of Theorem 3. We write them down here in a slightly different form than in \cite{F2}. Franssens proved the theorems by using generating functions but they can be proved also by induction, similarly as Theorem 1. 
\begin{theorem}
If functions $u=u(z)$ and $v=v(z)$ are any solutions of the equations (\ref{R}), (\ref{R2}) respectively, then 
the $n$th derivative of $v(z)$ is equal:
\begin{equation}\label{g}
	v^{(n)}(z)=v\frac{r^{n}}{2^{n}}\sum\limits_{k=1}^{n+1}M_{n+1,k}(u-a)^{n+1-k}(u-b)^{k-1}.
\end{equation}
\end{theorem}
Denote by $Q_{n}(u; a, b),\;n=0,1,2,\ldots $ the polynomial (of order $n$) standing on the right hand side of equation (\ref{g}) i.e.,
	\[Q_{n}(u; a, b)=\sum\limits_{k=1}^{n+1}M_{n+1,k}(u-a)^{n+1-k}(u-b)^{k-1}.
\]
\begin{theorem}
If functions $u=u(z)$ and $v=v(z)$ are any solutions of the equations (\ref{R}), (\ref{R2bis}) respectively, then 
the $n$th derivative of $v(z)$ is equal:
\begin{equation}\label{gd}
	v^{(n)}(z)=v\frac{r^{n}}{2^{n}}\sum\limits_{k=0}^{n}{n \choose k} (2d)^{k}Q_{n-k}(u; a, b).
\end{equation}
\end{theorem}
The polynomials $Q_{n}(u; a, b)$ are related to MacMahon polynomials (\ref{MMp}) by the formula
	\[M_{n}(x)=\frac{Q_{n}(u; a, b)}{(u-b)^{n}}\left|_{\frac{u-a}{u-b}=x}\right. .
\]
Similarly we denote by $P_{n+1}(u;a,b),\; n=1,2,\ldots$ the polynomial (of order $n+1$) standing of the right hand side of equation (\ref{f}). Thus
	\[P_{n+1}(u;a,b)=\sum\limits_{k=0}^{n-1} {n \bangle k}
(u-a)^{k+1}(u-b)^{n-k},\; n=1,2,\ldots \quad P_{1}(u)=u-a.
\]
Obviously the polynomial $P_{n+1}(u;a,b)$  can be rearranged into the Eulerian polynomial $E_{n}(x),\;\;n=1,2,\ldots$ using the following formula:
\begin{equation}
	E_{n}(x)=\frac{P_{n+1}(u;a,b)}{(u-b)^{n+1}}\left|_{\frac{u-a}{u-b}=x}\right.
\end{equation}
Polynomials $\{P_{n}(u;a,b)\}$ and $\{Q_{n}(u;a,b)\}$t are called the derivative polynomials. They have been introduced by Hoffman \cite{Ho} who used them to calculate some integrals with parameters and for summing some series, without giving any explicit formula for the coefficients. The polynomials were recently intensively studied (see for example \cite{Bo}, \cite{Bo2}, \cite{DC}, \cite{F}, \cite{Rz1}, \cite{Rz2}).
\section{Generating functions for the Eulerian polynomials}
It is easy to find the closed form of the following exponential generating function (see \cite{Ho}, \cite{F2}): 
\begin{equation}\label{gf}
	F(u,t)= u + rP_{2}(u;a,b)t+r^{2}P_{3}(u;a,b)\frac{t^2}{2!}+\cdots.
\end{equation}
For convenience of the reader we give the calculation for (\ref{gf}). Let $u=u(z)$  be a solution of the equation (\ref{R}). By the Taylor formula for the function $u=u(z)$ we have
{\setlength\arraycolsep{2pt}
\begin{eqnarray*}
F(u(z),t)&=& u(z) + rP_{2}(u(z);a,b)t+r^{2}P_{3}(u(z);a,b)\frac{t^2}{2!}+\cdots \\
&=& u(z)+u'(z)t+u''(z)\frac{t^2}{2!}+\cdots =u(z+t),
\end{eqnarray*}}
and 
\begin{equation}\label{gr1}
	F(u,t)= u(z(u)+t).
\end{equation}
For example if $a=0,\; b=1,\; u=1/(1+\exp(z)),\; \exp(z)=(1-u)/u$ (see point \ref{l}. on the list of functions on page 2) we get
\begin{equation}\label{gr2}
	F(u,t)= \frac{1}{1+e^{ z(u)+t}}=\frac{1}{1+\frac{1-u}{u}\:e^{t}}=\frac{u}{u+(1-u)e^{t}}.
\end{equation}
The generating function (\ref{gr2}) can be used for calculation of the exponential generating function for the Eulerian polynomials (\ref{Ep})
\begin{equation}\label{gr3}
	E_{0}(x) +E_{1}(x)y + E_{2}(x)\frac{y^2}{2!}+E_{3}(x)\frac{y^3}{3!}+\cdots .
\end{equation}
In order to do it let us observe, that we obtain the generating function (\ref{gr3}) by substituting in the expression
	\[\frac{F(u,t)-u}{u-1}+1,
\]
where $F(u,t)$ is given by the formula (\ref{gr2}), $u/(u-1)=x$ and $(u-1)t=y$ (that is $1/(u-1)=x-1$, $x/(x-1)=u$, $t=(x-1)y$). We calculate
{\setlength\arraycolsep{2pt}
\begin{eqnarray*}
&&\frac{F(u,t)-u}{u-1}+1=\frac{F(u,t)-1}{u-1}=\frac{\frac{u}{u+(1-u)e^{t}}-1}{u-1}=\frac{e^{t}}{u+(1-u)e^{t}}\\
&&=\frac{e^{(x-1)y}}{\frac{x}{x-1}-\frac{1}{x-1}e^{(x-1)y}}=\frac{1-x}{1-xe^{(1-x)y}}.
\end{eqnarray*}}
Therefore the generating function for the Eulerian polynomials is
\begin{equation}\label{gr4}
	E_{0}(x) +E_{1}(x)y + E_{2}(x)\frac{y^2}{2!}+E_{3}(x)\frac{y^3}{3!}+\cdots =\frac{1-x}{1-xe^{(1-x)y}}.
\end{equation}
Formula (\ref{gr4}) gives immediately the generating function for the Eulerian polynomials $\{A_{n}(x)\}$  defined by (\ref{Ep2}). We have
{\setlength\arraycolsep{2pt}
\begin{eqnarray}
&&A_{0}(x) +A_{1}(x)y + A_{2}(x)\frac{y^2}{2!}+A_{3}(x)\frac{y^3}{3!}+\cdots \nonumber\\
&& =\left(\frac{1-x}{1-xe^{(1-x)y}}-1 \right)\frac{1}{x} +1=\frac{x-1}{x-e^{(x-1)y}} \label{gr5}
\end{eqnarray}}
Foata \cite{Fo} notices that formula (\ref{gr5}) was known to Euler.
\section{Some others classical formulae concerning Eulerian polynomials}
The approach to Eulerian numbers and polynomials presented here is useful in obtaining other known results. For example the following classical formula concerning Eulerian polynomials
\begin{equation}\label{c1}
	E_{n}(x)=\sum\limits_{k=1}^{n-1}{n \choose k} E_{k}(x)(x-1)^{n-1-k}+E_{1}(x)(x-1)^{n-1}
\end{equation}
$n=1,2,\ldots $ or equivalently expressed in terms of the polynomials$\{A_{n}(x)\}$: 
\begin{equation}\label{c2}
	A_{n}(x)=\sum\limits_{k=0}^{n-1}{n \choose k} A_{k}(x)(x-1)^{n-1-k}
\end{equation} 
is an easy consequence of the following lemma.
\begin{lemma}
If a function $u=f(z)$ fulfills the equation $u'=u(u-1)$ then for any $n=1,2,\ldots$ 
\begin{equation}\label{fn}
	f^{(n)}(z)=(f(z)-1)\sum\limits_{k=0}^{n-1}{n \choose k}f^{(k)}(z).
\end{equation}
\end{lemma}
\begin{proof}
The proof is by induction with respect to $n$. For $n=1$ formula (\ref{fn}) is obviously true and let us suppose that it holds for a positive integer $n$. We have
{\setlength\arraycolsep{2pt}
\begin{eqnarray}
&&\hspace{-7mm}f^{(n+1)}(z)=f'(z)\sum\limits_{k=0}^{n-1}{n \choose k}f^{(k)}(z)+(f(z)-1)\sum\limits_{k=0}^{n-1}{n \choose k}f^{(k+1)}(z)\nonumber\\
&&=(f(z)\!-\!1)f(z)\sum\limits_{k=0}^{n-1}{n \choose k}f^{(k)}(z)+(f(z)\!-\!1)\sum\limits_{k=1}^{n}{n \choose k\!-\!1}f^{(k)}(z).\label{p1}
\end{eqnarray}}
By rearranging (\ref{fn}) to the form
	\[f(z)\sum\limits_{k=0}^{n-1}{n \choose k}f^{(k)}(z)=\sum\limits_{k=0}^{n}{n \choose k}f^{(k)}(z)
\]
and using it to the first sum of (\ref{p1}) we get
{\setlength\arraycolsep{2pt}
\begin{eqnarray*}	
f^{(n+1)}(z)&=&(f(z)-1)\sum\limits_{k=0}^{n}{n \choose k}f^{(k)}(z)+(f(z)-1)\sum\limits_{k=1}^{n}{n \choose k\!-\!1}f^{(k)}(z)\\
&=&(f(z)-1)\left(\sum\limits_{k=1}^{n}\left({n \choose k} + {n \choose k\!-\!1}\right)f^{(k)}(z)+f(z)\right)\\
&=&(f(z)-1)\sum\limits_{k=0}^{n}{n\!+\!1 \choose k}f^{(k)}(z),
\end{eqnarray*}}
and then formula (\ref{fn}) is proved.
\end{proof}
By using Theorem 1 we see that formula (\ref{fn}) is equivalent to
	\[\sum\limits_{j=0}^{n-1} {n \bangle j}u^{j+1}(u-1)^{n-j}=
	(u-1)\left(\sum\limits_{k=1}^{n-1} {n\choose k}\sum\limits_{j=0}^{k-1} {k \bangle j}u^{j+1}(u-1)^{k-j}+u\right).
\]
By substituting here $u/(u-1)=x$, $u=x/(x-1)$, $u-1=1/(x-1)$ we get
\[\sum\limits_{j=0}^{n-1} {n \bangle j}\frac{x^{j+1}}{(x-1)^{n+1}}=
	\frac{1}{x-1}\left(\sum\limits_{k=1}^{n-1} {n\choose k}\sum\limits_{j=0}^{k-1} {k \bangle j}\frac{x^{j+1}}{(x-1)^{k+1}}+\frac{x}{x-1}\right),
\]
hence we obtain the formula
	\[\frac{1}{ (x-1)^{n+1}}E_{n}(x)=\frac{1}{x-1}\left(\sum\limits_{k=1}^{n-1} {n\choose k}\frac{1}{(x-1)^{k+1}}E_{k}(x)+\frac{1}{x-1}E_{1}(x)\right),
\]
and the formulae (\ref{c1}) and (\ref{c2}) are proved.
\section{Generating functions for the MacMahon polynomials}
It is useful to get the generating function for the polynomials $Q_{n}(u)$ as
\[G(u,t)=Q_0(u;a,b)+\frac{r}{2}Q_{1}(u;a,b)t+\frac{r^{2}}{2^{2}}Q_{2}(u;a,b)\frac{t^{2}}{2!}+\frac{r^{3}}{2^{3}}Q_{3}(u;a,b)\frac{t^{3}}{3!}+\cdots
\]
Let the functions $u=u(z)$, $v=v(z)$ fulfill respectively equations (\ref{R}) and (\ref{R2}). Then using (\ref{g}) we have
{\setlength\arraycolsep{2pt}
\begin{eqnarray*}	
 v(z)G(u(z),t) &=& v(z)+v(z)\frac{r}{2}Q_{1}(u(z);a,b)t+v(z)\frac{r^{2}}{2^{2}}Q_{2}(v(z);a,b)\frac{t^{2}}{2!}+\cdots\\
& =& v(z)+v'(z)t+v''(z)\frac{t^{2}}{2!}+\cdots = g(z+t).
\end{eqnarray*}}
For example, for data from the point 5 on page 2 we get 
	\[\frac{e^{z/2}}{1+e^{z}}G(u(z), t)=\frac{e^{z/2}e^{t/2}}{1+e^{z}e^{t}},
\]
hence
	\[G(u(z), t)=\frac{(1+e^{z})e^{t/2}}{1+e^{z}e^{t}}.
\]
Therefore in this case
\begin{equation}\label{G}
	G(u, t)=\frac{(1+\frac{1-u}{u})e^{t/2}}{1+\frac{1-u}{u}e^{t}}=\frac{e^{t/2}}{u+(1-u)e^{t}}.
\end{equation}
The generating function (\ref{G}) can be used for calculation of the exponential generating function for the MacMahon polynomials (\ref{MMp})
\begin{equation}\label{G2}
	M_{0}(x) +\frac{1}{2}M_{1}(x)y + \frac{1}{2^{2}}M_{2}(x)\frac{y^2}{2!}+  \frac{1}{2^{3}}M_{3}(x)\frac{y^3}{3!}+\cdots .
\end{equation}
In order to do it let us observe, that we obtain the generating function (\ref{G2}) by substituting into 
$G(u,t)$ given by (\ref{G}), $u/(u-1)=x$ and $(u-1)t=y$ (that is $1/(u-1)=x-1$, $x/(x-1)=u$, $t=(x-1)y$). We calculate
	\[\frac{e^{t/2}}{u+(1\!-\!u)e^{t}}=\frac{e^{(x-1)y/2}}{\frac{x}{x-1}+(1\!-\!\frac{x}{x-1})e^{(x-1)y}}=
	\frac{(x\!-\!1)e^{(x-1)y/2}}{x\!-\!e^{(x-1)y/2}}=\frac{(1\!-\!x)e^{(1-x)y/2}}{1\!-\!xe^{(1-x)y/2}}.
\]
Thus
	\[M_{0}(x) +\frac{1}{2}M_{1}(x)y + \frac{1}{2^{2}}M_{2}(x)\frac{y^2}{2!}+  \frac{1}{2^{3}}M_{3}(x)\frac{y^3}{3!}+\cdots 
=\frac{(1\!-\!x)e^{(1-x)y/2}}{1\!-\!xe^{(1-x)y/2}}
\]
and
	\[M_{0}(x) +M_{1}(x)y + M_{2}(x)\frac{y^2}{2!}+  M_{3}(x)\frac{y^3}{3!}+\cdots 
=\frac{(1\!-\!x)e^{(1-x)y}}{1\!-\!xe^{(1-x)y}}.
\]
\section{Integral representations}
In paper \cite{Rz2} we have proved that for $n=1,2,3,\ldots$ 
\begin{equation}\label{i1}
	\int_{a}^{b}P_{n}(u;a,b)du=-(b-a)^{n+1}B_{n},
\end{equation}
where $B_{n}$ is the $n$th Bernoulli number and since $ P_{n}(u;a,b)$ is a polynomial i.e., an entire function, the integral can be understood as integral over any curve (piecewise smooth), joining points $a$ and $b$. Formula (\ref{i1}) is important because it gives immediately the following Grosset--Veselov formula (see Grosset--Veselov \cite{GV}) 
\begin{equation}\label{i2}
	B_{2m}=\frac{(-1)^{m-1}}{2^{2m+1}}\int_{-\infty}^{+\infty} \left(\frac{d^{m-1}}{dx^{m-1}}
	\frac{1}{\cosh^{2}x}\right)^{2}dx,
\end{equation}
which connects one--soliton solution of the KdV equation with Bernoulli numbers. Fairlie and Veselov \cite{FV} proved, by using the conservation laws, that KdV equation is directly related to the Faulhaber polynomials and the Bernoulli polynomials (see  \cite{Kn}). Grosset and Veselov \cite{GV} demonstrated the formula (\ref{i2}) in two ways, using the cited results and then adapting an idea due to Logan described in the book \cite{GKP}. Boyadzhiev \cite{Bo3} gave an alternative proof of (\ref{i2}), based on the Fourier transform. He noted that this proof was independently suggested by Professor A.
Staruszkiewicz (see also 'Note added in Proofs' at the end of \cite{GV}).

In order to prove (\ref{i2}) let us observe that one of the solutions of the equation (\ref{R}), for $a=-1, b=1, r=-1$, is $u=\tanh z$.
Since the image of the real line, under this function, is the interval $(-1,1)$ we have by (\ref{i1})
\begin{equation}\label{i3}
	\int_{-\infty}^{\infty}\frac{(\tanh z)^{(n-1)}}{\cosh^{2}z}\;dz=(-1)^{n-1}\int_{-1}^{1}P_{n}(u;-1,1)du=(-1)^{n}2^{n+1}B_{n}.
\end{equation}
Taking in (\ref{i3}) $n=2m$ and using, $(m-1)$-times, the formula of integration by parts for the leftmost integral, we get the Grosset--Veselov formula (\ref{i2}).\\
There arises a natural question about similar calculations for other polynomials e.g., for $Q_{n}(u,a,b)$. As we have proved in Sec. 4 (formula (\ref{G})) the generating function for the polynomials in the case of $a=0, b=1, r=1$ is
\[G(u,t)=Q_0(u;0,1))+\frac{1}{2}Q_{1}(u;0,1)t+\frac{1}{2^{2}}Q_{2}(u;0,1)\frac{t^{2}}{2!}+\cdots = \frac{e^{t/2}}{u+(1-u)e^{t}},
\]
and therefore
\begin{equation}\label{i4}
\int_{0}^{1}G(u,t)du	=e^{t/2}\int_{0}^{1}\frac{1}{u+(1-u)e^{t}}du=\frac{te^{t/2}}{e^{t}-1}.
\end{equation}
However since the generating function for the Bernoulli polynomials \\ $B_{0}(w), B_{1}(w),\ldots$ is
\begin{equation}\label{gB}
	B_{0}(w)+B_{1}(w)t+B_{2}(w)\frac{t^{2}}{2!}+B_{3}(w)\frac{t^{3}}{3!}+\cdots =\frac{te^{wt}}{e^{t}-1},
\end{equation}
then from (\ref{i4}) we get the following theorem.
\begin{theorem}
For $n=0,1,2,\ldots$ 
\begin{equation}\label{i5}
	\int_{0}^{1}Q_{n}(u;0,1)du=2^{n}B_{n}\left(\frac{1}{2}\right).
\end{equation}
\end{theorem} 
Since the polynomial $Q_{n}(u;a,b)$ is homogenous then by a suitable linear change of the variable in the integral (\ref{i5}) we get immediately
\begin{equation}\label{i6}
		\int_{a}^{b}Q_{n}(u;a,b)du=2^{n}B_{n}\left(\frac{1}{2}\right)(b-a)^{n+1}.
\end{equation}
Let us recall that $B_{n}=B_{n}(0)$. Then in view of (\ref{i1}) and (\ref{i6}) the next natural question arises, which concerns the existence of a family of polynomials 'connecting' polynomials $P_{n}(u;a,b)$ and $Q_{n}(u;a,b)$ in the sense that the corresponding integrals would give the values of the Bernoulli polynomial at intermediate points between $0$ and $\frac{1}{2}$. \\
Denote by $S_{n}(u; a,b,d),\;n=0,1,2,\ldots $ the polynomial (of order $n$) standing on the right hand side of the equation (\ref{gd}) i.e.,
\begin{equation}\label{gd2}
	S_{n}(u; a,b,d)=\sum\limits_{k=0}^{n}{n \choose k} (2d)^{k}Q_{n-k}(u; a, b).
\end{equation}
We will prove that $\{S_{n}(u; a,b,d)\}$ form the requested family of polynomials. A closed form formula for the following exponential generating function: 
\begin{equation}\label{gH}
	H(u,t)= S_{0}(u; a,b,d) + \frac{r}{2}S_{1}(u;a,b,d)t+\frac{r^{2}}{2^{2}}S_{2}(u;a,b,d)\frac{t^{2}}{2!}+\cdots,
\end{equation}
can be found similarly as in the previous cases. We assume that functions $u=u(z)$ and $v=v(z)$ are solutions of the equations (\ref{R}) and (\ref{R2bis}) respectively. Using the Taylor formula for the function $v=v(z)$ we have
{\setlength\arraycolsep{2pt}
\begin{eqnarray}
v(z)H(u(z),t)&=& v(z)S_{0}(u(z); a,b,d) + v(z)\frac{r}{2}S_{1}(u(z);a,b,d)t \nonumber \\
&&+ v(z)\frac{r^{2}}{2^{2}}S_{2}(u(z);a,b,d)\frac{t^{2}}{2!}\cdots \nonumber  \\ 
&=& v(z)+v'(z)t+v''(z)\frac{t^{2}}{2!}+\cdots =v(z+t).\label{g3}
\end{eqnarray}}
For example taking here $a=0, b=1, r=1$ and $\displaystyle u(z)=\frac{1}{1+e^{z}}$ the second equation (\ref{R2bis}) has the form
	\[v'(z)=v(z)\left(\frac{1}{1+e^{z}}-\frac{1}{2}+d\right),
\]
with a solution
\begin{equation}\label{s2}
	v(z)=\frac{e^{(1/2+d)z}}{1+e^{z}}.
\end{equation}
Using (\ref{g3}) and (\ref{s2}) we get
	\[H(u(z),t)=\frac{v(z+t)}{v(z)}=\frac{(1+e^{z})e^{(1/2+d)t}}{1+e^{z}e^{t}}
\]
and putting here $\displaystyle u(z)=\frac{1}{1+e^{z}}$, $\displaystyle e^{z}=\frac{1-u}{u}$ we arrive at
\begin{equation}\label{g4}
	H(u,t)=\frac{(1+\frac{1-u}{u})e^{(1/2+d)t}}{1+\frac{1-u}{u}e^{t}}=\frac{e^{(1/2+d)t}}{u+(1-u)e^{t}}.
\end{equation}
Then (\ref{g4}) yields
	\[\int_{0}^{1}H(u,t)du=\int_{0}^{1}\frac{e^{(1/2+d)t}}{u+(1-u)e^{t}}du=\frac{te^{(1/2+d)t}}{e^t-1},
\]
and therefore using (\ref{gB}) we arrive at the formula ($n=0,1,2,\ldots$)
 \begin{equation}\label{i7}
	\int_{0}^{1}S_{n}(u;0,1,d)du=2^{n}B_{n}\left(\frac{1}{2}+d\right).
\end{equation}
In order to generalize (\ref{i7}) to the polynomial $S_{n}(u;a,b,d)$ we use formula (\ref{gd2}). Therefore by (\ref{i6}) we have
{\setlength\arraycolsep{2pt}
\begin{eqnarray*}
&&\int_{a}^{b}S_{n}(u;a,b,d)=\sum\limits_{k=0}^{n}{n \choose k} (2d)^{k}\int_{a}^{b}Q_{n-k}(u; a, b)du \\
&&= \sum\limits_{k=0}^{n}{n \choose k} (2d)^{k} 2^{n-k}B_{n-k}\left(\frac{1}{2}\right)(b-a)^{n-k+1} \\
&&=2^{n}(b-a)^{n+1} \sum\limits_{k=0}^{n}{n \choose k} \left(\frac{d}{b-a}\right)^{k}B_{n-k}\left(\frac{1}{2}\right)\\
&&=2^{n}(b-a)^{n+1} B_{n}\left(\frac{1}{2}+\frac{d}{b-a}\right).
\end{eqnarray*}}
At the very end of the above calculation we have used the addition formula for the Bernoulli polynomials (see Temme \cite{T}, p.4)
	\[B_{n}(x+y)=\sum\limits_{k=0}^{n}{n \choose k}B_{k}(x)y^{n-k}.
\]
Thus we have proved the following
\begin{theorem}
For $n=1,2,\ldots $
	\[\int_{a}^{b}S_{n}(u;a,b,d)du=2^{n}(b-a)^{n+1} B_{n}\left(\frac{1}{2}+\frac{d}{b-a}\right).
\]
\end{theorem}
Comparing the generating functions (with parameters $a=0,\: b=1,\: r=1,\: d=-1/2$): $F(u,t)$ (given by formula (\ref{gr2})) with  $H(u,t)$  (formula (\ref{g4})) of the polynomials $\{P_{n}(u;0,1)\}$ and  $\{S_{n}(u;0,1,-1/2)\}$ respectively we get also
	\[\frac{P_{n+1}(u;0,1)}{u}=\frac{1}{2^{n}}S_{n}(u;0,1,-1/2).
\]
In particular, it follows that the coefficients of the polynomial \\
$\displaystyle \frac{1}{2^{n}}S_{n}(u;0,1,-1/2)$ are all integer.

\end{document}